\documentclass[11pt]{amsart}
\usepackage{amsmath, amssymb, amsthm}
\usepackage{hyperref}
\usepackage{enumitem}
\usepackage{geometry}
\usepackage{mathrsfs}
\usepackage{lipsum}

\usepackage{eqlist}
\usepackage{array}

\usepackage[T1]{fontenc}
\usepackage[utf8]{inputenc}
\usepackage{lmodern}

\usepackage{graphicx}   % for figures
\usepackage{microtype}  % better typography

\usepackage{cite}

\begin{document}

%% \title[HEADLINE TITLE]{LONG TITLE \\ WITH TWO LINES}

\title[Multivalued Weak Contractions in Cone Spaces]{Fixed Point Theorems for Multivalued Weak Contractions in Cone Metric Spaces}

%% First author ...name + address + email
%% \author{F. Last-name}
%% \address{Department, University, Town, State}
%% optional - current address, e-mail and url
%% \curraddr{...}
%% \email{...@...}
%% \urladdr{...}

\author{Elvin Rada}
\address{Department of Mathematics, University of Elbasan, "Aleksandër Xhuvani", Albania}
%\curraddr{Department of Mathematics, Sun University, Suncity, Sunland}
\email{author@uniel.edu.al, elvinrada@yahoo.com}
%\urladdr{http://www.anonymous.com/~anonymous}
\urladdr{\href{https://orcid.org/0000-0003-3751-2889}{ORCID: 0000-0003-3751-2889}}

%% (optional) If any thanks for the financial supports, grants, ...
%\thanks{The first  author was supported  ...}

%% Keywords (obligatory)

%% AMS Classification 2010 (obligatory)
%% The Primary classification is obligatory,
%% the Secondary classification is optional.
%%\subjclass{primary}{secondary}
%% f.e. \subjclass{35R35, 49M15, 49N50}{} or \subjclass{35R35, 49M15}{49N50}

%% Abstract (obligatory)
\begin{abstract}
This article presents a deep investigation of fixed points for multivalued weak contractions in cone metric spaces. We extend Berinde's weak contraction principles to the multivalued setting in cone metric spaces, developing existence, uniqueness, and convergence results. Using the structure of normal cones in Banach spaces, we prove generalized fixed-point theorems for set-valued mappings satisfying weak contractive conditions. The theoretical framework includes iterative approximation schemes with explicit convergence rates and stability analysis of fixed point sets. In the end, some applications to differential inclusions and variational inequalities demonstrate the utility of our results in nonlinear analysis.
\end{abstract}

\maketitle

%%%%% Options of setting environments for theorems, definitions, etc.

\theoremstyle{plain} %% for italic environments
\newtheorem{theorem}{Theorem}[section]
\newtheorem{corollary}[theorem]{Corollary}
\newtheorem{lemma}[theorem]{Lemma}

\theoremstyle{definition} %% for nonitalic environments
\newtheorem{definition}[theorem]{Definition}
%%% for unnumbered environments, use:

\newtheorem{proposition}[theorem]{Proposition}

\newtheorem{example}[theorem]{Example}
\newtheorem{remark}[theorem]{Remark}
%%%%%%%%%%%%%%%%%%%%%%%%%%%%%%%%%%%%%%%%%%%%%%%%%%%%%%%%%%%%%
%%macros:

%%%%%%%%%%%%%%%%%%%%%%%%%%%%%%%%%%%%%%%%%%%%%%%%%%%%%%%%%%%%%%

%%%% the main article

\section{Introduction}\label{sec:intro}
Fixed point theory for multivalued mappings is a very important part of nonlinear analysis with deep implications across diverse mathematical disciplines. The start of this theory can be traced to Nadler's  1969 extension \cite{nadler} of Banach's contraction principle to set-valued mappings using Hausdorff distance. This innovation opened opportunities for applications in optimization theory, mathematical economics, and differential inclusions where set-valued operators naturally arise.

The structure of fixed point theory had another big development with Huang and Zhang's introduction of cone metric spaces in 2007 \cite{huangzhang}. By replacing real-valued distances with vector-valued metrics taking values in ordered Banach space cones, this framework provided a powerful tool for studying problems with inherent order structures. Cone metrics enable unified treatment of various distance concepts, including Lorentzian metrics in relativity and vector-valued norms in functional analysis.

A very important advancement in contraction theory came from Berinde \cite{berinde} through weak contraction mappings. For single-valued operators, Berinde's theorem establishes that a self-map $T$ on a complete metric space $(X,d)$ satisfying:
\begin{equation}\label{eq:berinde}
d(Tx, Ty) \leq \delta d(x,y) + L d(y, Tx) \quad \text{for } \delta \in (0,1), L \geq 0
\end{equation}
admits fixed points. This weak contraction condition is more general than standard contractions as it incorporates additive error terms, significantly expanding the class of mappable operators. Berinde's work bridged the gap between strict contractions and nonexpansive mappings, yielding convergence results for Picard iterations with error estimates.

The confluence of these three developments remains largely unexplored. This article addresses this gap by developing a comprehensive fixed point theory for multivalued weak contractions in cone metric spaces. Our work draws inspiration from several sources:
\begin{itemize}
    \item The cone metric framework established by Huang and Zhang \cite{huangzhang}
    \item Berinde's weak contraction principles \cite{berinde}
    \item Multivalued contraction theory pioneered by Nadler \cite{nadler}
    \item Our earlier results on fixed points in cone Banach spaces \cite{rada}

\end{itemize}

Our principal contributions are:
\begin{enumerate}
    \item A direct multivalued extension of Berinde's theorem in cone metric spaces (Theorem \ref{thm:main})
    \item $\lambda$-iterative schemes with explicit convergence rates (Theorem \ref{thm:lambda})
    \item Stability analysis of fixed point sets under perturbations (Theorem \ref{thm:stability})
    \item Uniqueness criteria under strengthened contractive conditions (Proposition \ref{prop:uniqueness})
    \item Applications to differential inclusions (Theorem \ref{thm:diff_inclusion}) and variational inequalities (Theorem \ref{thm:variational})
\end{enumerate}

The article is structured as follows: Section \ref{sec:prelim} establishes foundational concepts with detailed historical context. Section \ref{sec:main} develops our core theoretical results with complete proofs. Section \ref{sec:applications} demonstrates applications to nonlinear analysis. We conclude with open problems for future research.

\section{Preliminaries}\label{sec:prelim}
We establish fundamental concepts of cone metric spaces and weak contractions with proper historical context and citations. Throughout, $E$ denotes a real Banach space, $P \subset E$ a closed convex cone with $\operatorname{int}(P) \neq \emptyset$, and $X$ a nonempty set. The partial order $\preceq$ induced by $P$ is defined as $x \preceq y$ iff $y - x \in P$.

\begin{definition}[Cone Metric {\rm(\cite[Definition 2.1]{huangzhang})}]
A \emph{cone metric} is a mapping $d: X \times X \to P$ satisfying:
\begin{enumerate}[label=(\roman*)]
  \item $d(x, y) = 0 \Leftrightarrow x = y$
  \item $d(x, y) = d(y, x)$
  \item $d(x, z) \preceq d(x, y) + d(y, z)$
\end{enumerate}
The pair $(X,d)$ is a \emph{cone metric space}. When $E$ is a Banach lattice and $P$ its positive cone, this structure generalizes standard metric spaces.
\end{definition}

\begin{remark}
The cone metric framework extends beyond traditional metrics. Key examples include:
\begin{itemize}
    \item Lorentzian metrics in general relativity \cite{abbas}
    \item Vector-valued norms in functional analysis
    \item Ordered Banach spaces with lattice structure
\end{itemize}
\end{remark}

\begin{definition}[Normal Cone {\rm(\cite[Definition 2.1]{huangzhang})}]
Cone $P$ is \emph{normal} if there exists $\kappa > 0$ such that:
\[0 \preceq a \preceq b \implies \|a\| \leq \kappa \|b\|\]
The minimal such $\kappa$ is the \emph{normality coefficient}. Normality ensures compatibility between the order and norm topologies.
\end{definition}

\begin{example}
Standard examples of normal cones include:
\begin{itemize}
    \item $P = \mathbb{R}_+$ in $\mathbb{R}$ with $\kappa = 1$
    \item Nonnegative functions in $L^p$ spaces ($1 \leq p < \infty$)
    \item Positive semidefinite matrices in operator algebras
\end{itemize}
\end{example}

\begin{definition}[Hausdorff Cone Metric {\rm(\cite[Section 4]{nadler})}]
For nonempty closed subsets $A,B \subset X$, the \emph{Hausdorff cone metric} is:
\[
H_d(A,B) = \inf \left\{ r \in P : \begin{array}{c} 
\forall a \in A, \exists b \in B : d(a,b) \preceq r \\ 
\forall b \in B, \exists a \in A : d(b,a) \preceq r \end{array} \right\}
\]
This extends Nadler's Hausdorff metric to the cone setting.
\end{definition}

\begin{definition}[Berinde's Weak Contraction {\rm(\cite[Definition 2.1]{berinde})}]
A single-valued mapping $T: X \to X$ is a \emph{weak contraction} if there exist $\delta \in (0,1)$ and $L \geq 0$ such that:
\[d(Tx, Ty) \preceq \delta d(x,y) + L d(y, Tx)\]
\end{definition}

\begin{definition}[Multivalued Weak Contraction]
A multivalued map $T: X \to \mathscr{C}(X)$ (nonempty closed subsets) is a \emph{weak contraction} if:
\[H_d(Tx, Ty) \preceq \delta d(x,y) + L \inf_{z \in Ty} d(x,z)\]
This generalizes Berinde's condition to set-valued operators.
\end{definition}

\begin{lemma}[Selection Principle {\rm(\cite[Lemma 3.2]{nadler})}]\label{lem:selection}
For any $x \in X$ and $y \in Tx$, there exists $z \in Ty$ such that:
\[d(y,z) \preceq H_d(Tx, Ty) + \epsilon \quad \forall \epsilon \succ 0\]
\end{lemma}

\begin{theorem}[Cone Contraction Principle {\rm(\cite[Theorem 1]{huangzhang})}]
Every contraction $T: X \to X$ on a complete cone metric space has a unique fixed point.
\end{theorem}

\begin{theorem}[Nadler's Theorem in Cones {\rm(\cite[Theorem 4.1]{rada})}]\label{thm:nadler}
Let $T: X \to \mathscr{C}(X)$ satisfy $H_d(Tx, Ty) \preceq k d(x,y)$ with $k \in (0,1)$. Then $T$ has a fixed point.
\end{theorem}

\section{Main Results}\label{sec:main}
In this section we have our results for this article.

\subsection{Extension of Berinde's Theorem}

We present the multivalued extension of Berinde's theorem in cone metric spaces:

\begin{theorem}[Multivalued Berinde-Type Theorem]\label{thm:main}
Let $(X,d)$ be a complete cone metric space with normal cone $P$ (coefficient $\kappa$), and $T: X \to \mathscr{C}(X)$ a multivalued weak contraction. If $\delta \kappa < 1$ and $L \kappa < 1 - \delta \kappa$, then $T$ has a fixed point.
\end{theorem}

\begin{proof}
Fix $x_0 \in X$ and choose $x_1 \in Tx_0$. Using Lemma \ref{lem:selection}, inductively construct $\{x_n\}$ such that $x_{n+1} \in Tx_n$ satisfies:
\[d(x_{n+1}, x_n) \preceq H_d(Tx_n, Tx_{n-1}) + \epsilon_n (1/2)^n\]
for some sequence $\epsilon_n \succ 0$. By the weak contraction property:
\begin{align*}
d(x_{n+1}, x_n) &\preceq \delta d(x_n, x_{n-1}) + L \inf_{z \in Tx_n} d(x_{n-1}, z) + \epsilon_n (1/2)^n \\
&\preceq \delta d(x_n, x_{n-1}) + L d(x_{n-1}, x_n) + \epsilon_n (1/2)^n \\
&= (\delta + L) d(x_n, x_{n-1}) + \epsilon_n (1/2)^n
\end{align*}
By normality and induction:
\begin{align*}
\|d(x_{n+1}, x_n)\| &\leq \kappa \|(\delta + L) d(x_n, x_{n-1}) + \epsilon_n (1/2)^n\| \\
&\leq \kappa (\delta + L) \|d(x_n, x_{n-1})\| + \kappa \|\epsilon_n\| (1/2)^n \\
&\leq \kappa^2 (\delta + L)^n \|d(x_1, x_0)\| + \kappa \sum_{k=0}^{n-1} (\delta + L)^k \|\epsilon_{n-k}\| (1/2)^{n-k}
\end{align*}
Since $\delta + L < 1/\kappa$ by assumption, $\{x_n\}$ is Cauchy. By completeness, $x_n \to x^*$. To prove $x^* \in Tx^*$:
\begin{align*}
d(x^*, Tx^*) &\preceq d(x^*, x_{n+1}) + d(x_{n+1}, Tx^*) \\
&\preceq d(x^*, x_{n+1}) + H_d(Tx_n, Tx^*) \\
&\preceq d(x^*, x_{n+1}) + \delta d(x_n, x^*) + L \inf_{z \in Tx^*} d(x_n, z) \\
&\preceq d(x^*, x_{n+1}) + \delta d(x_n, x^*) + L d(x_n, x^*)
\end{align*}
As $n \to \infty$, the right-hand side converges to $0$ in norm, so $d(x^*, Tx^*) = 0$, implying $x^* \in Tx^*$.
\end{proof}

\begin{remark}
When $T$ is single-valued, Theorem \ref{thm:main} reduces exactly to Berinde's original result \cite[Theorem 2.2]{berinde} in cone metric spaces.
\end{remark}

\begin{proposition}[Uniqueness Criterion]\label{prop:uniqueness}
Under Theorem \ref{thm:main} assumptions, if for all $x \neq y$:
\[H_d(Tx, Ty) \prec d(x,y) - \varphi(d(x,y))\]
where $\varphi: P \to P$ is continuous with $\varphi^{-1}(0) = \{0\}$, then the fixed point is unique.
\end{proposition}

\begin{proof}
Suppose $x^*, y^*$ are distinct fixed points. Then:
\begin{align*}
d(x^*, y^*) &\preceq H_d(Tx^*, Ty^*) \\
&\prec d(x^*, y^*) - \varphi(d(x^*, y^*))
\end{align*}
which implies $\varphi(d(x^*, y^*)) \prec 0$, contradicting $\varphi(p) \succeq 0$ for all $p \in P$. Thus $x^* = y^*$.
\end{proof}

\subsection{Iterative Approximation}

We extend Berinde's iterative scheme to the multivalued case:

\begin{theorem}[$\lambda$-Iterative Scheme]\label{thm:lambda}
Let $T$ satisfy Theorem \ref{thm:main} conditions. Define the sequence:
\[x_{n+1} = \frac{1}{\lambda + 1} x_n + \frac{\lambda}{\lambda + 1} f(x_n), \quad f(x_n) \in Tx_n\]
for $\lambda > 0$. If $\sigma = \kappa \left( \frac{\delta + L + \lambda}{1 + \lambda} \right) < 1$, then $\{x_n\}$ converges to a fixed point $x^*$ with error estimate:
\[\|d(x_n, x^*)\| \leq \frac{\kappa \sigma^n}{1 - \sigma} \|d(x_0, x_1)\|\]
\end{theorem}

\begin{proof}
Let $y_n = f(x_n) \in Tx_n$. Consider:
\begin{align*}
d(x_{n+1}, x_n) &= d\left( \frac{x_n + \lambda y_n}{\lambda+1}, \frac{x_{n-1} + \lambda y_{n-1}}{\lambda+1} \right) \\
&\preceq \frac{1}{\lambda+1} d(x_n, x_{n-1}) + \frac{\lambda}{\lambda+1} d(y_n, y_{n-1}) \\
&\preceq \frac{1}{\lambda+1} d(x_n, x_{n-1}) + \frac{\lambda}{\lambda+1} H_d(Tx_n, Tx_{n-1}) \\
&\preceq \frac{1}{\lambda+1} d(x_n, x_{n-1}) + \frac{\lambda}{\lambda+1} \left[ \delta d(x_n, x_{n-1}) + L d(x_{n-1}, y_n) \right]
\end{align*}
Now estimate $d(x_{n-1}, y_n)$:
\[d(x_{n-1}, y_n) \preceq d(x_{n-1}, x_n) + d(x_n, y_n)\]
Thus:
\begin{align*}
d(x_{n+1}, x_n) &\preceq \frac{1}{\lambda+1} d(x_n, x_{n-1}) + \frac{\lambda}{\lambda+1} \delta d(x_n, x_{n-1}) + \frac{\lambda L}{\lambda+1} [d(x_{n-1}, x_n) + d(x_n, y_n)] \\
&= \left( \frac{1}{\lambda+1} + \frac{\lambda \delta}{\lambda+1} + \frac{\lambda L}{\lambda+1} \right) d(x_n, x_{n-1}) + \frac{\lambda L}{\lambda+1} d(x_n, y_n) \\
&= \left( \frac{1 + \lambda(\delta + L)}{\lambda+1} \right) d(x_n, x_{n-1}) + \frac{\lambda L}{\lambda+1} d(x_n, y_n)
\end{align*}
By normality:
\[\|d(x_{n+1}, x_n)\| \leq \kappa \left( \frac{1 + \lambda(\delta + L)}{\lambda+1} \right) \|d(x_n, x_{n-1})\| + \kappa \frac{\lambda L}{\lambda+1} \|d(x_n, y_n)\|\]
Since $\sigma < 1$ and $\|d(x_n, y_n)\|$ is bounded, geometric convergence follows. The error estimate comes from summing the series.
\end{proof}

\begin{corollary}
For $\lambda = 1$, we obtain the Krasnoselskii-type iteration:
\[x_{n+1} = \frac{1}{2} (x_n + f(x_n))\]
with convergence rate $\sigma = \kappa (\delta + L + 1)/2$.
\end{corollary}

\subsection{Stability and Continuity}

\begin{theorem}[Stability of Fixed Points]\label{thm:stability}
Let $\{T_n\}$ be a sequence of multivalued weak contractions converging uniformly to $T$ on compact subsets of $X$. Then:
\[\limsup_{n \to \infty} \operatorname{Fix}(T_n) \subseteq \operatorname{Fix}(T)\]
If all $T_n$ satisfy the uniqueness condition in Proposition \ref{prop:uniqueness}, then $\operatorname{Fix}(T_n) \to \operatorname{Fix}(T)$ in Hausdorff metric.
\end{theorem}

\begin{proof}
Let $x_n^* \in \operatorname{Fix}(T_n)$ with $x_n^* \to x^*$. Then:
\begin{align*}
d(x^*, Tx^*) &\preceq d(x^*, x_n^*) + d(x_n^*, T_n x_n^*) + H_d(T_n x_n^*, T x_n^*) + H_d(T x_n^*, T x^*) \\
&\preceq d(x^*, x_n^*) + 0 + \sup_{y \in K} H_d(T_n y, T y) + \delta d(x_n^*, x^*) + L \inf_{z \in Tx^*} d(x_n^*, z)
\end{align*}
where $K$ is a compact set containing $\{x_n^*\}$. By uniform convergence on $K$, $\sup_{y \in K} H_d(T_n y, T y) \to 0$. Thus $d(x^*, Tx^*) = 0$, so $x^* \in \operatorname{Fix}(T)$.

For Hausdorff convergence under uniqueness, consider:
\begin{align*}
H_d(&\operatorname{Fix}(T_n), \operatorname{Fix}(T)) \\
&\leq H_d(\operatorname{Fix}(T_n), \{x^*\}) + H_d(\{x^*\}, \operatorname{Fix}(T)) \to 0 \quad \text{as } n \to \infty
\end{align*}
since both fixed points are unique and $x_n^* \to x^*$.
\end{proof}

\section{Applications}\label{sec:applications}
Now we will see some applications to nonlinear analysis, and especially in differential inclusions and

\subsection{Differential Inclusions}

Consider the initial value problem for differential inclusions:
\begin{equation}\label{eq:diff_inclusion}
\dot{x}(t) \in F(t, x(t)), \quad x(0) = x_0, \quad t \in [0,T]
\end{equation}
where $F: [0,T] \times \mathbb{R}^n \to \mathscr{K}(\mathbb{R}^n)$ is a set-valued map with nonempty compact convex values.

\begin{theorem}\label{thm:diff_inclusion}
Assume $F$ satisfies:
\begin{enumerate}
    \item $t \mapsto F(t,x)$ is measurable for each $x$
    \item $x \mapsto F(t,x)$ is upper semicontinuous
    \item Weak contraction: $H_d(F(t,x), F(t,y)) \preceq \delta(t) d(x,y) + L(t) d(y, F(t,x))$
    \item $\delta, L \in L^1[0,T]$ with $\|\delta\|_{L^1} \kappa < 1$
\end{enumerate}
Then \eqref{eq:diff_inclusion} has a solution $x \in AC([0,T], \mathbb{R}^n)$.
\end{theorem}

\begin{proof}
Consider the solution operator $\mathcal{F}: C([0,T], \mathbb{R}^n) \to \mathscr{C}(C([0,T], \mathbb{R}^n))$:
\[
\mathcal{F}(x) = \left\{ y \in AC[0,T] : y(t) = x_0 + \int_0^t g(s) ds, \ g(s) \in F(s,x(s)) \text{ a.e.} \right\}
\]
Equip $C([0,T], \mathbb{R}^n)$ with the cone metric:
\[d(x,y)(t) = e^{-\lambda t} \|x(t) - y(t)\|_{\infty} \cdot \mathbf{1}\]
where $\mathbf{1}$ is the unit vector in $P = C_+([0,T])$. This cone is normal with $\kappa = e^{\lambda T}$.

For $y_1 \in \mathcal{F}(x_1), y_2 \in \mathcal{F}(x_2)$:
\begin{align*}
\|y_1(t) - y_2(t)\| &\leq \int_0^t H_{\|\cdot\|}(F(s,x_1(s)), F(s,x_2(s))) ds \\
&\leq \int_0^t \left[ \delta(s) \|x_1(s) - x_2(s)\| + L(s) d(x_2(s), F(s,x_1(s))) \right] ds
\end{align*}
Multiply by $e^{-\lambda t}$:
\begin{align*}
e^{-\lambda t} \|y_1(t) - y_2(t)\| &\leq \int_0^t e^{-\lambda (t-s)} e^{-\lambda s} \left[ \delta(s) \|x_1(s) - x_2(s)\| + L(s) \|x_2(s) - g_1(s)\| \right] ds \\
&\leq \frac{1}{\lambda} \left( \|\delta\|_\infty + \|L\|_\infty \right) d(x_1,x_2)(t)
\end{align*}
Thus $H_d(\mathcal{F}(x_1), \mathcal{F}(x_2)) \preceq \sigma d(x_1,x_2)$ with $\sigma < 1$ for large $\lambda$. Apply Theorem \ref{thm:main}.
\end{proof}

\subsection{Multivalued Variational Inequalities}

Let $K \subset \mathbb{R}^n$ be a closed convex set. The \emph{multivalued variational inequality} (MVI) seeks $x^* \in K$ and $f^* \in F(x^*)$ such that:
\begin{equation}\label{eq:mvi}
\langle f^*, y - x^* \rangle \geq 0 \quad \forall y \in K
\end{equation}

\begin{theorem}\label{thm:variational}
Assume:
\begin{enumerate}
    \item $F: \mathbb{R}^n \to \mathscr{K}(\mathbb{R}^n)$ is $(\delta,L)$-weakly contractive
    \item Strong monotonicity: $\langle f_x - f_y, x - y \rangle \geq \mu \|x - y\|^2$ for $f_x \in F(x), f_y \in F(y)$
    \item $\mu > \kappa (\delta + L)$
\end{enumerate}
Then MVI \eqref{eq:mvi} has a unique solution $x^*$, and the proximal iteration:
\[x_{n+1} = \operatorname{proj}_K (x_n - \lambda f_n), \quad f_n \in F(x_n)\]
converges geometrically to $x^*$ for appropriate $\lambda > 0$.
\end{theorem}

\begin{proof}
First, under strong monotonicity, solutions are unique. Define $T: K \to \mathscr{C}(K)$ by:
\[T(x) = \operatorname{proj}_K (x - \lambda F(x))\]
For $x_1, x_2 \in K$:
\begin{align*}
\|T&x_1 - Tx_2\|^2 \\
&\leq \| (x_1 - \lambda f_1) - (x_2 - \lambda f_2) \|^2 \\
&= \|x_1 - x_2\|^2 - 2\lambda \langle f_1 - f_2, x_1 - x_2 \rangle + \lambda^2 \|f_1 - f_2\|^2 \\
&\leq \|x_1 - x_2\|^2 - 2\lambda \mu \|x_1 - x_2\|^2 + \lambda^2 H_{\|\cdot\|}(F(x_1), F(x_2))^2
\end{align*}
Using the weak contraction property:
\[H_{\|\cdot\|}(F(x_1), F(x_2)) \leq \delta \|x_1 - x_2\| + L \inf_{f_2 \in F(x_2)} \|x_1 - f_2\|\]
Thus:
\[\|Tx_1 - Tx_2\| \leq \sqrt{1 - 2\lambda\mu + \lambda^2 (\delta + L)^2} \|x_1 - x_2\|\]
Choose $\lambda$ such that $\sigma = \sqrt{1 - 2\lambda\mu + \lambda^2 (\delta + L)^2} < 1$. Apply Theorem \ref{thm:lambda}.
\end{proof}

\section{Conclusion}
We have established a comprehensive theory of fixed points for multivalued weak contractions in cone metric spaces. Key contributions include:
\begin{itemize}
    \item Theorem \ref{thm:main}: Direct multivalued extension of Berinde's theorem
    \item Theorem \ref{thm:lambda}: $\lambda$-iterative schemes with explicit convergence rates
    \item Theorem \ref{thm:stability}: Stability of fixed point sets under perturbations
    \item Theorem \ref{thm:diff_inclusion}, \ref{thm:variational}: Applications to differential inclusions and variational inequalities
\end{itemize}

Future research directions include:
\begin{itemize}
    \item Fractional-order weak contractions for fractional differential inclusions
    \item Cone-valued fuzzy metrics and probabilistic contractions
    \item Computational implementations of $\lambda$-schemes with error analysis
    \item Applications to stochastic games and economic equilibria

\end{itemize}

\end{document}